\documentclass{amsart} 

\usepackage{amsmath,amssymb,latexsym, amscd}
\usepackage{exscale, cite, eps fig, graphics}

\renewcommand{\geq}{\geqslant}
\renewcommand{\leq}{\leqslant}

\newtheorem{theorem}{Theorem}[section]
\newtheorem{lemma}[theorem]{Lemma}

\newtheorem*{main-theorem}{Main Theorem}
\newtheorem*{remark*}{Remark}
\numberwithin{equation}{section}

\title[Wave breaking for fractional dispersion]
{Wave breaking for the Whitham equation \\ with fractional dispersion}

\author[Hur]{Vera~Mikyoung~Hur}
\address{Department of Mathematics, University of Illinois at Urbana-Champaign, Urbana, IL 61801 USA}
\email{verahur@math.uiuc.edu}

\author[Tao]{Lizheng~Tao}
\email{leedstao@illinois.edu}  

\begin{document}

\maketitle

\begin{abstract}
We show wave breaking for the Whitham equation in a range of fractional dispersion, 
i.e. the solution remains bounded but its slope becomes unbounded in finite time,
provided that the initial datum is sufficiently steep. 
\end{abstract}

\section{Introduction}\label{sec:intro}

We study finite-time blowup for the Whitham equation with fractional dispersion
\begin{equation}\label{E:main}
\partial_tu+\mathcal{H}\Lambda^\alpha u+u\partial_xu=0
\end{equation}
in the range $\alpha>0$, where $t\in\mathbb{R}$ denotes the temporal variable, 
$x\in\mathbb{R}$ is the spatial variable, and $u=u(x,t)$ is real-valued.
Moreover, $\mathcal{H}$ denotes the Hilbert transform
and $\Lambda=\sqrt{-\partial_x^2}$ is a Fourier multiplier, defined via its symbol as 
\[
\widehat{\Lambda f}(\xi)=|\xi|\widehat{f}(\xi).
\]
For $0<\alpha<1$ (see \cite{Hur-breaking}, for instance)
\begin{equation}\label{def:Lambda}
\mathcal{H}\Lambda^\alpha f(x)= \int^\infty_{-\infty} \frac{\text{sgn}(x-y)}{|x-y|^{1+\alpha}}(f(x)-f(y))~dy
\end{equation}
up to multiplication by a constant.

\

In the case of $\alpha=3$, most notably, \eqref{E:main} makes the Korteweg-de Vries (KdV) equation, 
and in the case of $\alpha=2$, it corresponds to the Benjamin-Ono equation. 
In the case of $\alpha=1/2$, moreover, \eqref{E:main} was argued in \cite{Hur-breaking} 
to have relevances to surface water waves in two dimensions in the infinite depths.
In particular, it shares in common with the physical problem 
the dispersion relation and scaling symmetry. 
Furthermore \eqref{E:main} belongs to the family of nonlinear nonlocal\footnote{
Note that \eqref{E:whitham} is nonlocal unless $m(\xi)$ is a polynomial of $i\xi$.} equations
(see \cite{NS94}, for instance)
\begin{equation}\label{E:whitham}
\partial_tu-\mathcal{M}\partial_xu+u\partial_xu=0\qquad 
\text{where}\quad\widehat{\mathcal{M}f}(\xi)=m(\xi)\widehat{f}(\xi),
\end{equation}
describing wave motions in various physical situations.
Note that \eqref{E:whitham} combines the nonlinearity, 
which compels singularities in short intervals of time,
and dispersion, which instead acts to spread out waves and make them decay over time. 

\

The global well-posedness for \eqref{E:main} is adequately understood 
in the case of $\alpha=3$ and $\alpha=2$. 
Recently there have been increasing research activities 
about regularity versus finite-time blowup for $1<\alpha<2$; 
see \cite{LPS} and references therein. 

\

The KdV equation well explains long wave phenomena in a channel of water, 
e.g. solitary waves, but it loses relevances for short and intermediately long waves.
In particular, waves in shallow water at times develop a vertical slope or a multi-valued profile
whereas the KdV equation prevents singularity formation from solutions.
Whitham (see \cite{Whitham}) therefore proposed \eqref{E:whitham}, for which 
$m(\xi)=\sqrt{\tanh\xi\,/\xi}$,  
as an alternative to the KdV equation, combining the full range\footnote{
Note that $\sqrt{\tanh\xi/\xi}$ is the phase speed of a plane wave with the spatial frequency $\xi$
near the quintessential state of water. Since 
\[
\sqrt{\tanh\xi/\xi}=1-\tfrac16\xi^2+O(\xi^4) \qquad \text{for $\xi\ll 1$},
\]
one may regard the KdV equation (after normalization of parameters) 
as to approximate up to second order the dispersion of the Whitham equation, 
and hence the water wave problem, in the long wavelength regime.}
of the dispersion of surface water waves and the nonlinearity of the shallow water equations.
Moreover Whitham advocated that it would capture {\em wave breaking},
i.e. the solution remains bounded but its slope becomes unbounded in finite time. 
Wave breaking was analytically confirmed in \cite{NS94} and \cite{CE98}, for instance, 
for \eqref{E:whitham}, provided that 
$\widehat{m} \in L^1$ away from zero and it satisfies a few other assumptions. 
Note however that the kernel associated with the integral representation of\footnote{
Note that $\mathcal{H}\Lambda^\alpha=-\partial_x\Lambda^{\alpha-1}$.} $\Lambda^{\alpha-1}$ 
is $|\cdot|^{-\alpha}$ up to multiplication by a constant, 
and hence the proofs in \cite{NS94} and \cite{CE98} may not be applicable 
for \eqref{E:main} in the range $0<\alpha<1$. 
Nevertheless, recently in \cite{CCG} and \cite{Hur-breaking}, 
a H\"older-norm blowup was shown for \eqref{E:main} for $0<\alpha<1$. 
Incidentally gradient blowup was shown in \cite{DDL09} and \cite{KNS08}, for instance,
for the Whitham equation with fractional diffusion (compare to \eqref{E:main})
\[
\partial_tu+\Lambda^\alpha u+u\partial_xu=0
\]
for $0<\alpha<1$. 

\

Here we promote the result in \cite{CCG} and \cite{Hur-breaking} to wave breaking
in the range $0<\alpha<1/2$, provided that the initial datum is in a Gevrey class and sufficiently steep.

\begin{theorem}[Wave breaking for $0<\alpha<1/2$] \label{thm:1/2}
Let $0<\alpha <{\displaystyle \frac{(1-\epsilon)^2}{3(1+\epsilon)^3-(1-\epsilon)^2}}$ 
for $\epsilon>0$ sufficiently small. 
Assume that $\phi \in H^\infty(\mathbb{R})$ satisfies 
\begin{gather}
\epsilon^2(\inf_{x\in\mathbb{R}}\phi'(x))^2>
\|\phi\|_{H^3}+\frac{2}{\alpha}\Big(3C_1+\frac{C_2}{1-\alpha}\Big)\label{A2:m1} \\
-\epsilon(1-\epsilon)^3\inf_{x\in\mathbb{R}}\phi'(x)>\frac{2}{\alpha(1-\alpha)}
\Big(3+\frac{1}{1-\alpha}\Big(\frac{C_1}{C_0}+\frac{C_2}{C_1}\Big)\Big),\label{A2:m2} \\
-\epsilon\frac{1+\epsilon}{1-\epsilon}\inf_{x\in\mathbb{R}}\phi'(x)>
\frac{6}{\alpha}(1+\epsilon^{1/\alpha}),\label{A2:m3}
\end{gather}
where
\begin{gather} 
\|\phi\|_{L^\infty}+\|\phi'\|_{L^\infty}<(1-\epsilon)C_0, \label{I:C0} \\
\|\phi'\|_{L^\infty}<(1-\epsilon)C_1<-(1+\epsilon)\inf_{x\in\mathbb{R}}\phi'(x), \label{I:C1} \\
0<\frac{1-\epsilon}{1+\epsilon}C_2<
-\frac{1/\alpha-1}{e}\Big(\frac23\Big)^{1/\alpha-1}\inf_{x\in\mathbb{R}}\phi'(x) \label{I:C2}
\end{gather}
and
\begin{equation}\label{I:Gevrey}
\|\phi^{(n)}\|_{L^\infty}<C_2(n-1)^{(n-1)/\alpha}\qquad \text{for $n=2, 3, \dots$}.
\end{equation}
Then the solution of the initial value problem associated with \eqref{E:main} and $u(x,0)=~\phi(x)$ 
exhibits wave breaking, i.e.
\[ 
|u(x,t)|<\infty \qquad \text{for all $x\in\mathbb{R}$}\quad\text{for all $t\in [0,T)$}
\]
but
\[
\inf_{x\in\mathbb{R}}\partial_xu(x,t) \to -\infty \qquad\text{as $t\to T-$}
\] 
for some $T>0$. Moreover
\begin{equation}\label{E:T}
-\frac{1}{\inf_{x\in\mathbb{R}}\phi'(x)}\frac{1}{1+\epsilon}<T<
-\frac{1}{\inf_{x\in\mathbb{R}}\phi'(x)}\frac{1}{(1-\epsilon)^2}.
\end{equation}
\end{theorem}

Furthermore we show wave breaking in the range $0<\alpha<1/3$
for a broad class of initial data with sufficiently large steepness.

\begin{theorem}[Wave breaking for $0<\alpha<1/3$] \label{thm:1/3}
Let ${\displaystyle 0<\alpha<\frac{(1-\epsilon)^2}{5-2(1-\epsilon)^2}}$ 
for $\epsilon>~0$ sufficiently small. Assume that $\phi\in H^\infty(\mathbb{R})$ satisfies
\begin{equation}\label{A3:m1}
\epsilon^2(\inf_{x\in\mathbb{R}}\phi'(x))^2>\|\phi\|_{H^2}+\frac{1}{\alpha}(6C_1+C_2),
\end{equation}
and 
\begin{equation}\label{A3:m2}
-\inf_{x\in\mathbb{R}}\phi'(x)>\frac{4}{\alpha(1-\alpha)}\Big(3+\frac{1}{1-\alpha}\frac{C_1}{C_0}\Big)
+\frac{2}{\alpha}\Big(6+\frac{C_2}{C_1}\Big),
\end{equation}
where
\begin{equation}\label{I:C012}
\|\phi\|_{L^\infty}+\|\phi'\|_{L^\infty}<\frac12C_0, \quad \|\phi'\|_{L^\infty}<\frac12C_1,
\quad\|\phi''\|_{L^2}<\sqrt{\frac{1-2\alpha}{2}}C_2.
\end{equation}
Then the solution of the initial value problem associated with \eqref{E:main} and $u(x,0)=~\phi(x)$ 
exhibits wave breaking.
\end{theorem}

The proofs follow those in \cite{NS94}, based upon 
ordinary differential equations with nonlocal forcing terms along characteristics (see \eqref{eqn:vn}).
The kernel associated with their integral representations is explicit, albeit singular (see \eqref{def:Kn}), 
and they are bounded by Sobolev norms of the solution.
We simplify some arguments in \cite{NS94} in the course of the proofs. 

\

Recent numerical studies in \cite{KS} (see also \cite{KZ}) suggest 
finite-time blowup for \eqref{E:main} for\footnote{
Note that \eqref{E:main} preserves the $L^2$-norm.
Note moreover that \eqref{E:main} remains invariant under 
\[
u(x,t)\mapsto \lambda^{\alpha-1}u(\lambda x, \lambda^\alpha t)
\]
for any $\lambda>0$, whence it is $\dot{H}^{3/2-\alpha}$-critical. 
}
$0\leq\alpha\leq3/2$ (and global regularity for $\alpha>3/2$),
but the blowup scenario be different from wave breaking for $1<\alpha\leq3/2$. 
In the case of $\alpha=0$, a H\"older-norm blowup was shown in \cite{CCG}, and 
in the case of $\alpha=1$, incidentally, \eqref{E:main} becomes the Burgers equation, 
which exhibits wave breaking.
It will be interesting to analytically confirm blowup in the range $0\leq\alpha\leq3/2$. 
Moreover it will be interesting to confirm wave breaking in the range $0\leq\alpha\leq1$. 

\section{Proof of Theorem \ref{thm:1/2}}\label{sec:1/2}

We assume that the initial value problem associated with \eqref{E:main} and $u(x,0)=\phi(x)$ 
possesses a unique solution in the class $C^\infty([0,T); H^\infty(\mathbb{R}))$ for some $T>0$. 
Using a priori energy bounds and the method of successive approximations, indeed, 
one may establish local in time well-posedness for \eqref{E:main} in $H^{3/2+}(\mathbb{R})$. 
We assume that $T$ is the maximal time of existence.

\

For $x\in\mathbb{R}$ let $X(t;x)$ solve
\[
\frac{dX}{dt}(t;x)=u(X(t;x),t), \qquad X(0;x)=x.
\]
Since $u(x,t)$ is bounded and satisfies a Lipschitz condition in $x$ 
for all $x\in\mathbb{R}$ for all $t\in[0,T)$, it follows from the ODE theory 
that $X(\cdot\,;x)$ exists throughout the interval $[0,T)$ for all $x\in\mathbb{R}$. 
Furthermore $x\mapsto X(\cdot\,;x)$ is continuously differentiable 
throughout the interval $(0,T)$ for all $x\in\mathbb{R}$.

Let 
\begin{equation}\label{def:vn}
v_n(t;x)=(\partial_x^nu)(X(t;x),t) \qquad \text{for $n=0,1,2,\dots$}
\end{equation}
and let
\begin{equation}\label{def:m}
m(t)=\inf_{x\in\mathbb{R}}v_1(t;x)=\inf_{x\in\mathbb{R}}(\partial_xu)(x,t)=:m(0)q^{-1}(t).
\end{equation}
Clearly 
\[
m(t)<0\text{ for all $t\in[0,T)$},\quad q(0)=1\quad\text{and}\quad q(t)>0\text{ for all $t\in[0,T)$.}
\] 
Indeed $m(t)\geq 0$ would imply that $u(\cdot,t)$ be non-decreasing in $\mathbb{R}$, 
and hence~$u(\cdot,t)\equiv~0$. 

\

For $x\in\mathbb{R}$, differentiating \eqref{E:main} with respect to $x$ and evaluating at $x=X(t;x)$, 
we arrive at that 
\begin{align}
\frac{dv_n}{dt}+\sum_{j=1}^n\left(\begin{matrix}n\\j\end{matrix}\right)v_jv_{n+1-j}+&K_n(t;x)=0
\qquad \text{for $n=2,3,\dots$},\label{eqn:vn} \\
\frac{dv_1}{dt}+v_1^2+K_1(t;x)&=0 \label{eqn:v1} 
\intertext{and, similarly,}
\frac{dv_0}{dt}+K_0(t;x)=&0, \label{eqn:v0}
\end{align}
where $\left(\begin{matrix}n\\j\end{matrix}\right)$ means a binomial coefficient 
and (see~\eqref{def:Lambda})
\begin{align}\label{def:Kn}
K_n(t;x)=&(\mathcal{H}\Lambda^\alpha\partial_x^nu)(X(t;x),t) \notag\\
=&\int^\infty_{-\infty} \frac{\text{sgn}(X(t;x)-y)}{|X(t;x)-y|^{1+\alpha}}
((\partial_x^nu)(X(t;x),t)-(\partial_x^nu)(y,t))~dy
\end{align}
for $n=0,1,2,\dots$. 
Let $\delta>0$. We split the integral and perform an integration by parts to show that
\begin{align}\label{eq:Kn}
|K_n(t;x)|=&\Big|\Big(\int_{|y|<\delta}+\int_{|y|>\delta}\Big)
\frac{\text{sgn}(y)}{|y|^{1+\alpha}}((\partial_x^nu)(X(t;x),t)-(\partial_x^nu)(X(t;x)-y,t))~dy\Big|\notag \\
=&\Big|\frac{1}{\alpha}\delta^{-\alpha}
((\partial_x^nu)(X(t;x)-\delta,t)-(\partial_x^nu)(X(t;x)+\delta,t)) \notag \\
\qquad&+\frac{1}{\alpha}\int_{|y|<\delta}\frac{1}{|y|^\alpha} (\partial_x^{n+1}u)(X(t;x)-y,t)~dy \notag \\
\qquad&+\int_{|y|>\delta}\frac{\text{sgn}(y)}{|y|^{1+\alpha}}
((\partial_x^nu)(X(t;x),t)-(\partial_x^nu)(X(t;x)-y,t))~dy\Big|\notag \\
\leq &\frac{6}{\alpha}\delta^{-\alpha}\|v_n(t)\|_{L^\infty}
+\frac{2}{\alpha(1-\alpha)}\delta^{1-\alpha}\|v_{n+1}(t)\|_{L^\infty}
\end{align}
for $n=0,1,2,\dots$ and for all $t\in [0,T)$ for all $x\in\mathbb{R}$.
Indeed 
${\displaystyle \frac{\text{sgn}(y)}{|y|^{1+\alpha}}=-\frac{1}{\alpha}\Big(\frac{1}{|y|^{\alpha}}\Big)'}$. 
We pause to remark that the kernel associated with 
the integral representation of $\mathcal{H}\Lambda^\alpha$ in the range $0<\alpha<1$ 
is {\em not} in $L^1$ near zero whereas its anti-derivative is. 

\

We shall show that 
\begin{equation}\label{claim:K1}
|K_1(t;x)|<\epsilon^2m^2(t)\qquad \text{for all $t\in [0,T)$ for all $x\in \mathbb{R}$}.
\end{equation}
Note from \eqref{A2:m1} and the Sobolev inequality that 
\[
|K_1(0;x)|=|\mathcal{H}\Lambda^\alpha\phi'(x)|\leq \|\phi\|_{H^{\alpha+3/2+}}<\epsilon^2m^2(0)
\qquad\text{for all $x\in\mathbb{R}$}.
\]
Suppose on the contrary that $|K_1(T_1;x)|=\epsilon^2m^2(T_1)$ 
for some $T_1\in(0,T)$ for some $x\in\mathbb{R}$. 
By continuity, without loss of generality, we may assume that
\begin{equation}\label{I:K1}
|K_1(t;x)|\leq \epsilon^2m^2(t) \qquad \text{for all $t\in [0,T_1]$ for all $x\in \mathbb{R}$}.
\end{equation}

\begin{lemma}\label{lem:S}
For $t\in[0,T_1]$ let 
\[
\Sigma(t)=\{x\in\mathbb{R}:v_1(t;x)\leq (1-\epsilon)m(t)\}.
\]
Then $\Sigma(t_2)\subset\Sigma(t_1)$ whenever $0\leq t_1\leq t_2\leq T_1$.
\end{lemma}

\begin{proof}
The proof may be found in \cite[Lemma~2.6.3]{NS94}. Here we include the detail for completeness.

Suppose on the contrary that $x_1\notin\Sigma(t_1)$ but $x_1\in\Sigma(t_2)$
for some $x_1\in\mathbb{R}$ for some $0\leq t_1\leq t_2\leq T_1$, i.e.
\begin{equation}\label{def:x1}
v_1(t_1;x_1)>(1-\epsilon)m(t_1)\quad\text{and}\quad
v_1(t_2;x_1)\leq(1-\epsilon)m(t_2)<\frac12m(t_2).
\end{equation}
We may choose $t_1$ and $t_2$ close so that 
\[
v_1(t;x_1)\leq\frac12m(t)\qquad\text{for all $t\in[t_1,t_2]$.}
\]
Indeed $v_1(\cdot\,;x_1)$ and $m$ are uniformly continuous throughout the interval $[0,T_1]$. Let 
\begin{equation}\label{def:x2}
v_1(t_1;x_2)=m(t_1)<\frac12m(t_1).
\end{equation}
We may necessarily choose $t_2$ close to $t_1$ so that 
\[
v_1(t;x_2)\leq\frac12m(t)\qquad \text{for all $t\in[t_1,t_2]$.}
\]
Consequently \eqref{I:K1} yields that 
\[
|K_1(t;x_j)|\leq \epsilon^2m^2(t)\leq 4\epsilon^2v_1^2(t;x_j)<\frac12\epsilon v_1^2(t;x_j)
\qquad \text{for all $t\in[t_1,t_2]$ and $j=1,2$.}
\]
To proceed, we use \eqref{eqn:v1} to show that 
\[
\frac{dv_1}{dt}(\cdot\,;x_1)=-v_1^2(\cdot\,;x_1)-K_1(\cdot\,;x_1)
\geq\Big(-1-\frac{\epsilon}{2}\Big)v_1^2(\cdot\,;x_1)
\]
and, similarly,
\[
\frac{dv_1}{dt}(\cdot\,;x_2)\leq\Big(-1+\frac{\epsilon}{2}\Big)v_1^2(\cdot\,;x_2)
\]
throughout the interval $(t_1,t_2)$. It then follows after integration that
\[
\hspace*{-10pt}v_1(t_2;x_1)\geq\frac{v_1(t_1;x_1)}{1+(1+\frac{\epsilon}{2})v_1(t_1;x_1)(t_2-t_1)}
\quad\text{and}\quad
v_1(t_2;x_2)\leq\frac{v_1(t_1;x_2)}{1+(1-\frac{\epsilon}{2})v_1(t_1;x_2)(t_2-t_1)}.
\]
The latter inequality and \eqref{def:x2} imply that 
\[
m(t_2)\leq\frac{m(t_1)}{1+(1-\frac{\epsilon}{2})m(t_1)(t_2-t_1)}.
\]
The former inequality and \eqref{def:x1}, on the other hand, imply that
\begin{align*}
v_1(t_2;x_1)>&\frac{(1-\epsilon)m(t_1)}{1+(1+\frac{\epsilon}{2})(1-\epsilon)m(t_1)(t_2-t_1)} \\
>&\frac{(1-\epsilon)m(t_1)}{1+(1-\frac{\epsilon}{2})m(t_1)(t_2-t_1)} \\
\geq& (1-\epsilon)m(t_2).
\end{align*}
A contradiction therefore completes the proof.\end{proof}

\begin{lemma}\label{lem:q}
$0<q(t)\leq 1$ for all $t\in[0,T_1]$.
\end{lemma}

\begin{proof}
Let $x\in\Sigma(T_1)$ and we shall suppress it to simplify the exposition.
Note from Lemma~\ref{lem:S} that 
\begin{equation}\label{I:mv}
m(t)\leq v_1(t)\leq (1-\epsilon)m(t)\qquad \text{for all $t\in[0,T_1]$}.
\end{equation}
Let's write the solution of \eqref{eqn:v1} as
\begin{equation}\label{def:r}
v_1(t)=\frac{v_1(0)}{1+v_1(0)\int^t_0(1+(v_1^{-2}K_1)(\tau))~d\tau}=:m(0)r^{-1}(t).
\end{equation}
Clearly $r(t)>0$ for all $t\in[0,T_1]$. Since 
\[
|(v_1^{-2}K_1)(t)|<(1-\epsilon)^{-2}\epsilon^2<\epsilon\qquad\text{for all $t\in[0,T_1]$}
\]
by \eqref{I:mv} and \eqref{I:K1}, we use \eqref{def:r} to show that 
\begin{equation}\label{I:dr/dt}
(1+\epsilon)m(0)\leq \frac{dr}{dt}\leq (1-\epsilon)m(0)\qquad\text{}.
\end{equation}
throughout the interval $(0,T_1)$.
Consequently $r(t)$, and hence (see \eqref{def:r}) $v_1(t)$, are decreasing for all $t\in[0,T_1]$. 
Furthermore (see \eqref{def:m}), $m(t)$, and hence $q(t)$, are decreasing for all $t\in[0,T_1]$.
This completes the proof. Incidentally note from \eqref{def:m}, \eqref{def:r}, and \eqref{I:mv} that
\begin{equation}\label{I:r}
q(t)\leq r(t)\leq \frac{1}{1-\epsilon}q(t)\qquad\text{for all $t\in[0,T_1]$}.
\end{equation} 
\end{proof}

\begin{lemma}\label{lem:qs}
In case $s>0$, $s\neq1$,
\begin{align}
\int^t_0q^{-s}(\tau)~d\tau\leq&
-\frac{1}{(1-\epsilon)^{s+1}}\frac{1}{m(0)}\frac{1}{1-s}((1-\epsilon)^{s-1}-q^{1-s}(t)),
\label{I:s>1}
\intertext{and}
\int^t_0q^{-1}(\tau)~d\tau\leq& -\frac{1}{(1-\epsilon)^2}\frac{1}{m(0)}
\Big(\log\frac{1}{1-\epsilon}-\log q(t)\Big)\label{I:s=1}
\end{align}
for all $t\in[0,T_1]$.
\end{lemma}

\begin{proof}
In case $s>0$, $s\neq 1$, we use \eqref{I:r} and \eqref{I:dr/dt} to show that 
\begin{align*}
\int^t_0 q^{-s}(\tau)~d\tau\leq& \frac{1}{(1-\epsilon)^s}\int^t_0r^{-s}(\tau)~d\tau \\
\leq &\frac{1}{(1-\epsilon)^{s+1}}\frac{1}{m(0)}\int^t_0r^{-s}(\tau)\frac{dr}{d\tau}(\tau)~d\tau \\
=&\frac{1}{(1-\epsilon)^{s+1}}\frac{1}{m(0)}\frac{1}{1-s}(r^{1-s}(t)-r^{1-s}(0)).
\end{align*}
Therefore \eqref{I:s>1} follows from \eqref{I:r}. Similarly
\[
\int^t_0q^{-1}(\tau)~d\tau\leq \frac{1}{(1-\epsilon)^2}\frac{1}{m(0)}(\log r(t)-\log r(0)),
\]
and \eqref{I:s=1} follows from \eqref{I:r}.
\end{proof}

We claim that
\begin{align}
\|v_0(t)\|_{L^\infty}&=\|u(t)\|_{L^\infty}<C_0, \label{claim:v0} \\
\|v_1(t)\|_{L^\infty}&=\|\partial_xu (t)\|_{L^\infty}<C_1q^{-1}(t), \label{claim:v1}  \\
\|v_n(t)\|_{L^\infty}&=\|\partial_x^nu (t)\|_{L^\infty}<
C_2(n-1)^{(n-1)/\alpha}q^{-1-(n-1)\sigma}(t),\label{claim:vn}
\end{align}
$n=2,3,\dots$, for all $t\in[0,T_1]$,
where $C_0, C_1, C_2$ satisfy \eqref{I:C0}, \eqref{I:C1}, \eqref{I:C2} and
\begin{equation}\label{I:sigma}
\sigma>3\frac{(1+\epsilon)^3}{(1-\epsilon)^2}-1
\end{equation}
so that $\sigma\alpha<1$ (see Theorem~\ref{thm:1/2}).
Note from \eqref{I:C0}, \eqref{I:C1}, \eqref{I:Gevrey} and \eqref{def:m} that
\begin{align*}
\|v_0(0)\|_{L^\infty}=&\|\phi\|_{L^\infty}<C_0,\\
\|v_1(0)\|_{L^\infty}=&\|\phi'\|_{L^\infty}<C_1=C_1q^{-1}(0), \\
\|v_n(0)\|_{L^\infty}=&\|\phi^{(n)}\|_{L^\infty}<C_2(n-1)^{(n-1)/\alpha}q^{-1-(n-1)\sigma}(0),
\end{align*}
$n=2,3,\dots$. 
Suppose on the contrary that \eqref{claim:v0}, \eqref{claim:v1} and \eqref{claim:vn} hold
for all $n=0,1,2,\dots$ throughout the interval $[0,T_2)$ 
but it fail for some $n\geq0$ at $t=T_2$ for some $T_2 \in (0,T_1]$. 
By continuity, 
\begin{align}
\|v_0(t)\|_{L^\infty}\leq &C_0, \label{I:v0}\\ 
\|v_1(t)\|_{L^\infty}\leq &C_1q^{-1}(t), \label{I:v1}\\
\|v_n(t)\|_{L^\infty}\leq &C_2(n-1)^{(n-1)/\alpha}q^{-1-(n-1)\sigma}(t), \label{I:vn}
\end{align}
$n=2,3,\dots$, for all $t\in[0,T_2]$.

\

For $n=0$, we use \eqref{eq:Kn}, where $\delta(t)=q(t)$, and \eqref{I:v0}, \eqref{I:v1} to show that 
\[
|K_0(t;x)|\leq \frac{6}{\alpha}C_0q^{-\alpha}(t)+\frac{2}{\alpha(1-\alpha)}C_1q^{1-\alpha}(t)q^{-1}(t)
=\frac{2}{\alpha}\Big(3C_0+\frac{C_1}{1-\alpha}\Big)q^{-\alpha}(t)
\]
for all $t\in[0,T_2]$ for all $x\in\mathbb{R}$. 
We then integrate \eqref{eqn:v0} over the interval $[0,T_2]$ to arrive at that
\begin{align*}
|v_0(T_2;x)|\leq& \|\phi\|_{L^\infty}+\int^{T_2}_0|K_0(t;x)|~dt \\
<&(1-\epsilon)C_0+\frac{2}{\alpha}\Big(3C_0+\frac{C_1}{1-\alpha}\Big)\int^{T_2}_0q^{-\alpha}(t)~dt \\
<&(1-\epsilon)C_0-\frac{2}{\alpha(1-\alpha)}\Big(3C_0+\frac{C_1}{1-\alpha}\Big)
\frac{1}{(1-\epsilon)^{\alpha+1}}\frac{1}{m(0)}((1-\epsilon)^{\alpha-1}-q^{1-\alpha}(T_2)) \\
<&(1-\epsilon)C_0-\frac{2}{\alpha(1-\alpha)}\Big(3C_0+\frac{C_1}{1-\alpha}\Big)
\frac{1}{(1-\epsilon)^2}\frac{1}{m(0)}\\<&C_0
\end{align*}
for all $x\in\mathbb{R}$.
Therefore \eqref{claim:v0} holds throughout $[0,T_2]$.
Here, the second inequality uses \eqref{I:C0}, the third inequality uses \eqref{I:s>1},
the fourth inequality uses Lemma~\ref{lem:q}, and the last inequality uses \eqref{A2:m2}.

\

For $n=1$, similarly, we use \eqref{eq:Kn}, where $\delta(t)=q^\sigma(t)$, 
and \eqref{I:v1}, \eqref{I:vn} to show that
\begin{align}
|K_1(t;x)|\leq& \frac{6}{\alpha}C_1q^{-1}q^{-\sigma\alpha}(t)
+\frac{2}{\alpha(1-\alpha)}C_2q^{\sigma-\sigma\alpha}(t)q^{-1-\sigma}(t) \notag \\
=&\frac{2}{\alpha}\Big(3C_1+\frac{C_2}{1-\alpha}\Big)q^{-1-\sigma\alpha}(t) 
<\frac{2}{\alpha}\Big(3C_1+\frac{C_2}{1-\alpha}\Big)q^{-2}(t) \,. \label{I:K1'}
\end{align}
for all $t\in[0,T_2]$ for all $x\in\mathbb{R}$. 
The last inequality uses Lemma~\ref{lem:q} and that $\sigma\alpha<1$.
Note from \eqref{eqn:v1} that 
\[
\frac{dv_1}{dt}=-v_1^2-K_1(t;x)\leq |K_1(t;x)|.
\]
In case $v_2(T_2;x)\geq 0$, we integrate it over the interval $[0,T_2]$ to arrive at that
\begin{align*}
v_1(T_2;x)\leq& \|\phi'\|_{L^\infty}+\int^{T_2}_0|K_1(t;x)|~dt \\
\leq&(1-\epsilon)C_1+\frac{2}{\alpha}\Big(3C_1+\frac{C_2}{1-\alpha}\Big)\int^{T_2}_0q^{-2}(t)~dt \\
\leq&(1-\epsilon)C_1q^{-1}(T_2)-\frac{2}{\alpha}\Big(3C_1+\frac{C_2}{1-\alpha}\Big)
\frac{1}{(1-\epsilon)^3}\frac{1}{m(0)}(q^{-1}(T_2)-(1-\epsilon)) \\
<&(1-\epsilon)C_1q^{-1}(T_2)
-\frac{2}{\alpha}\Big(3C_1+\frac{C_2}{1-\alpha}\Big)\frac{1}{(1-\epsilon)^3}\frac{1}{m(0)}q^{-1}(T_2)\\ 
<&C_1q^{-1}(T_2).
\end{align*}
The second inequality uses \eqref{I:C1} and \eqref{I:K1'}, 
the third inequality uses Lemma~\ref{lem:q} and \eqref{I:s>1}, 
and the last inequality uses \eqref{A2:m2}.
In case $v_1(T_2;x)<0$, on the other hand, we may assume, without loss of generality, that 
$\|\phi'\|_{L^\infty}=-m(0)$, and \eqref{def:m} and \eqref{I:C1} imply that 
\[
v_1 (T_2;x) \geq m(T_2)= m(0) q^{-1}(T_2)> -C_1 q^{-1}(T_2).
\]
Therefore \eqref{claim:v1} holds throughout the interval $[0,T_2]$.

\

To proceed, for $n\geq 2$, we use \eqref{eq:Kn}, where $\delta(t)=n^{-1/\alpha}q^\sigma(t)$, 
and \eqref{I:vn} to show that
\begin{align}\label{I:K2}
|K_n(t;x)|\leq &\frac{6}{\alpha}C_2n(n-1)^{(n-1)/\alpha}q^{-\sigma\alpha}(t)q^{-1-(n-1)\sigma}(t)\notag\\
&+\frac{2}{\alpha(1-\alpha)}C_2n^{1-1/\alpha}n^{n/\alpha}
q^{\sigma-\sigma\alpha}q^{-1-n\sigma}(t) \notag \\
\leq&\max\Big(\frac{6}{\alpha}, \frac{2}{\alpha(1-\alpha)}\Big)
C_2n(n-1)^{(n-1)/\alpha}\Big(1+\Big(\frac{n}{n-1}\Big)^{(n-1)/\alpha}\Big)
q^{-1-\sigma\alpha-(n-1)\sigma}(t)\notag \hspace*{-.5in} \\
<&\frac{6}{\alpha}C_2n(n-1)^{(n-1)/\alpha}(1+e^{1/\alpha})q^{-1-\sigma\alpha-(n-1)\sigma}(t)
\end{align}
for all $t\in[0,T_2]$ for all $x\in\mathbb{R}$. 

\

For $n\geq 2$, furthermore, let $v_1(T_3;x)=m(T_3)$ and 
\[
v_1(t;x)\leq \frac{1}{(1+\epsilon)^{1/(2+(n-1)\sigma)}}m(t)\qquad\text{for all $t\in[T_3,T_2]$}
\]
for some $T_3\in(0,T_2)$ and for some $x\in\mathbb{R}$. 
Indeed $v_1$ and $m$ are uniformly continuous throughout the interval $[0, T_2]$.
We rerun the argument in the proof of Lemma~\ref{lem:q} to arrive at that
\[
(1+\epsilon)m(0)\leq\frac{dr}{dt}(t)\leq(1-\epsilon)m(0) \qquad \text{for all $t\in (T_3,T_2)$}
\]
and
\[
q(t)\leq r(t)\leq (1+\epsilon)^{1/(2+(n-1)\sigma)}q(t)\qquad\text{for all $t\in[T_3,T_2]$.}
\]
Moreover we rerun the argument in the proof of Lemma~\ref{lem:qs} to arrive at that
\begin{align}
\int^{T_2}_{T_3}&q^{-2-(n-1)\sigma}(t)~dt \notag\\ 
&\leq(1+\epsilon)\int^{T_2}_{T_3}r^{-2-(n-1)\sigma}(t)~dt \notag\\
&\leq\frac{1+\epsilon}{1-\epsilon}\frac{1}{m(0)}
\int^{T_2}_{T_3}r^{-2-(n-1)\sigma}(t)\frac{dr}{dt}(t)~dt \notag \\
&\leq-\frac{1+\epsilon}{1-\epsilon}\frac{1}{m(0)}\frac{1}{1+(n-1)\sigma}
(r^{-1-(n-1)\sigma}(T_2)-r^{-1-(n-1)\sigma}(T_3)) \notag \\
&\leq-\frac{1+\epsilon}{1-\epsilon}\frac{1}{m(0)}\frac{1}{1+(n-1)\sigma}
(q^{-1-(n-1)\sigma}(T_2)-q^{-1-(n-1)\sigma}(T_3)).\label{I:s>1'}
\end{align}

\

For $n=2$, note from \eqref{eqn:vn} that 
\begin{align*}
\frac{dv_2}{dt}=&-3v_1v_2-K_2(t;x) \\ \leq& 3C_1C_2q^{-1}(t)q^{-1-\sigma}(t)
+\frac{12}{\alpha}(1+e^{1/\alpha})C_2q^{-1-\sigma\alpha-\sigma}(t) \\
\leq&3\Big(C_1+\frac{4}{\alpha}(1+e^{1/\alpha})\Big)C_2q^{-2-\sigma}(t)
\end{align*}
throughout the interval $(0,T_2)$. The first inequality uses \eqref{I:vn} and \eqref{I:K2}, 
and the second inequality uses Lemma~\ref{lem:q} and that $\sigma\alpha<1$.
Let $v_2(T_2;x_2)=\max_{x\in\mathbb{R}}|v_2(T_2;x)|$. 
We may choose $T_3$ close to $T_2$ so that 
$v_2(t;x_2)\geq 0$ for all $t\in[T_3,T_2]$. 
We may furthermore assume, without loss of generality, that $\|\phi'\|_{L^\infty}=-m(0)$. 
An integration then leads to that
\begin{align*}
v_2(T_2;x_2)\leq&v_2(T_3;x_2)
+3\Big(C_1+\frac{4}{\alpha}(1+e^{1/\alpha})\Big)C_2\int^{T_2}_{T_3}q^{-2-\sigma}(t)~dt\\
<&C_2q^{-1-\sigma}(T_3)\\
&-3\Big(C_1+\frac{4}{\alpha}(1+e^{1/\alpha})\Big)C_2
\frac{1}{1+\sigma}\frac{1+\epsilon}{1-\epsilon}\frac{1}{m(0)}
(q^{-1-\sigma}(T_2)-q^{-1-\sigma}(T_3)) \\
<&C_2q^{-1-\sigma}(T_3)\\&+\frac{3}{1+\sigma}\frac{1+\epsilon}{1-\epsilon}
\Big(\frac{1+\epsilon}{1-\epsilon}+\epsilon\frac{1+\epsilon}{1-\epsilon}\Big)C_2
(q^{-1-\sigma}(T_2)-q^{-1-\sigma}(T_3))  \\
=&\Big( 1- \frac{3}{1+\sigma}\frac{(1+\epsilon)^3}{(1-\epsilon)^2}\Big)C_2 q^{-1-\sigma}(T_3)
+ \frac{3}{1+\sigma}\frac{(1+\epsilon)^3}{(1-\epsilon)^2}C_2 q^{-1-\sigma} (T_2) \\
\leq & \Big( 1- \frac{3}{1+\sigma}\frac{(1+\epsilon)^3}{(1-\epsilon)^2}\Big)C_2 q^{-1-\sigma} (T_2)
+ \frac{3}{1+\sigma}\frac{(1+\epsilon)^3}{(1-\epsilon)^2}C_2 q^{-1-\sigma} (T_2) \\
=&C_2q^{-1-\sigma}(T_2).
\end{align*}
Therefore \eqref{claim:vn} holds for $n=2$ throughout the interval $[0,T_2]$.
Here, the second inequality uses \eqref{I:vn} and \eqref{I:s>1'}, 
the third inequality uses \eqref{I:C1} and \eqref{A2:m3}, 
and the last inequality uses~\eqref{I:sigma} and that $q(t)$ is decreasing for all $t\in[T_3,T_2]$
(see the proof of Lemma~\ref{lem:q}).

\begin{lemma}\label{lem:stirling}
For $n\geq 3$, 
\begin{equation}\label{I:stirling}
\sum_{j=2}^{n-1}\left(\begin{matrix}n\\j\end{matrix}\right)(j-1)^{(j-1)/\alpha}(n-j)^{(n-j)/\alpha}
\leq \frac{e}{1/\alpha-1}\Big(\frac32\Big)^{1/\alpha-1}n(n-1)^{(n-1)/\alpha}.
\end{equation}
\end{lemma}

The proof may be found in \cite[Lemma~2.6.1]{NS94}. Hence we omit the detail.


\

For $n\geq 3$, note from \eqref{eqn:vn} that 
\begin{align*}
\frac{dv_n}{dt}=&
-(n+1)v_1v_n-\sum_{j=2}^{n-1}\left(\begin{matrix}n\\j\end{matrix}\right)v_jv_{n+1-j}-K_n(t;x) \\
\leq &C_1C_2(n+1)(n-1)^{(n-1)/\alpha}q^{-1}(t)q^{-1-(n-1)\sigma}(t) \\
&+\sum_{j=2}^{n-1}C_2^2\left(\begin{matrix}n\\j\end{matrix}\right)
(j-1)^{(j-1)/\alpha}(n-j)^{(n-j)/\alpha}q^{-1-(j-1)\sigma}(t)q^{-1-(n-j)\sigma}(t)+K_n(t;x) \\
\leq&C_1C_2(n+1)(n-1)^{(n-1)/\alpha}q^{-2-(n-1)\sigma}(t)\\
&+C_2^2 \frac{e^{1/\alpha}}{1/\alpha-1}n(n-1)^{(n-1)/\alpha}q^{-2-(n-1)\sigma}(t)\\
&+\frac{6}{\alpha}(1+e^{1/\alpha})C_2n(n-1)^{(n-1)/\alpha}q^{-2-(n-1)\sigma}(t) \\
\leq &\Big(C_1(n+1)+\frac{e}{1/\alpha-1}\Big(\frac32\Big)^{1/\alpha-1}C_2n
+\frac{6}{\alpha}(1+e^{1/\alpha})n\Big)\\ 
&\hspace*{160pt} \times C_2(n-1)^{(n-1)/\alpha}q^{-2-(n-1)\sigma}(t)
\end{align*}
throughout the interval $(0,T_2)$. 
The first inequality uses \eqref{I:v1} and \eqref{I:vn}, 
and the second inequality uses \eqref{I:stirling} and \eqref{I:K2}. 
Let $v_n(T_2;x_n)=\max_{x\in\mathbb{R}}|v_n(T_2;x)|$. 
We may choose $T_3$ close to $T_2$ so that $v_n(t;x_n)\geq 0$ for all $t\in[T_3,T_2]$. 
An integration then leads to that
\begin{align*}
v_n(T_2;x_n)\leq &v_n(T_3;x_n)+\Big(C_1(n+1)
+\frac{e}{1/\alpha-1}\Big(\frac32\Big)^{1/\alpha-1}C_2n+\frac{6}{\alpha}(1+e^{1/\alpha})n\Big) \\
&\hspace{110pt}\times C_2(n-1)^{(n-1)/\alpha}\int^{T_2}_{T_3}q^{-2-(n-1)\sigma}(t)~dt \\
<&C_2(n-1)^{(n-1)/\alpha}q^{-1-(n-1)\sigma}(T_3) \\ &-\Big(C_1(n+1)
+\frac{e}{1/\alpha-1}\Big(\frac32\Big)^{1/\alpha-1}C_2n+\frac{6}{\alpha}(1+e^{1/\alpha})n\Big)\\
&\hspace*{70pt}\times\frac{1}{1+(n-1)\sigma}\frac{1+\epsilon}{1-\epsilon}\frac{1}{m(0)}C_2(n-1)^{(n-1)/\alpha}\\ 
&\hspace*{100pt}\times (q^{-1-(n-1)\sigma}(T_2)-q^{-1-(n-1)\sigma}(T_3)) \\
<&C_2(n-1)^{(n-1)/\alpha}q^{-1-(n-1)\sigma}(T_3) \\
&+\frac{1}{2\sigma+1}\frac{1+\epsilon}{1-\epsilon}\Big(4\frac{1+\epsilon}{1-\epsilon}
+\frac{1+\epsilon}{1-\epsilon}+\epsilon\frac{1+\epsilon}{1-\epsilon}\Big)C_2(n-1)^{(n-1)/\alpha} \\
&\hspace*{100pt}\times (q^{-1-(n-1)\sigma}(T_2)-q^{-1-(n-1)\sigma}(T_3)) \\
=&\Big(1-\frac{1}{2\sigma+1}\frac{(5+\epsilon)(1+\epsilon)^2}{(1-\epsilon)^2}\Big)
C_2(n-1)^{(n-1)/\alpha}q^{-1-(n-1)\sigma}(T_3) \\
&+\frac{1}{2\sigma+1}\frac{(5+\epsilon)(1+\epsilon)^2}{(1-\epsilon)^2}
C_2(n-1)^{(n-1)/\alpha}q^{-1-(n-1)\sigma}(T_2) \\
\leq &C_2(n-1)^{(n-1)/\alpha}q^{-1-(n-1)\sigma}(T_2).
\end{align*}
Therefore \eqref{claim:vn} holds for $n\geq 3$ throughout the interval $[0,T_2]$. 
Here, the second inequality uses \eqref{I:vn} and \eqref{I:s>1'}, 
the third inequality uses \eqref{I:C2}, \eqref{A2:m3} 
and that $n\geq 3$, and the last inequality uses \eqref{I:sigma} and 
that $q(t)$ is decreasing for all $t\in[T_3,T_2]$. 

\

To summarize, a contradiction proves that \eqref{claim:v0}, \eqref{claim:v1} and \eqref{claim:vn} hold
for all $n=0,1,2,\dots$ throughout the interval $[0,T_1]$. 

\

To proceed, note from \eqref{I:K1'}, \eqref{def:m} and \eqref{A2:m1} that 
\[
|K_1(t;x)|\leq \frac{2}{\alpha}\Big(3C_1+\frac{C_2}{1-\alpha}\Big)m^{-2}(0)m^2(t)<\epsilon^2m^2(t)
\]
for all $t\in [0,T_1]$ for all $x\in\mathbb{R}$.
A contradiction therefore proves \eqref{claim:K1}.
We remark that \eqref{claim:v0}, \eqref{claim:v1}, \eqref{claim:vn} hold for all $n=0,1,2,\dots$ 
throughout the interval $[0, T']$ for all $T'<T$.

\

To conclude, for $t\in [0,T)$ let $x \in \Sigma(t)$. We use \eqref{def:r} and \eqref{I:dr/dt} to show that 
\[
m(0)(v_1^{-1}(0;x)+(1+\epsilon) t) \leq r(t;x)\leq m(0)(v_1^{-1}(0;x)+(1-\epsilon) t).
\]
Lemma~\ref{lem:S} moreover implies that $m(0)<v_1(0;x)\leq (1-\epsilon)m(0)$. 
Consequently
\[
1+m(0)(1+\epsilon)t \leq r(t) \leq \frac{1}{1-\epsilon} + m(0) (1-\epsilon)t
\]
Furthermore \eqref{I:r} implies that 
\[
(1-\epsilon) + m(0)(1-\epsilon^2)t \leq q(t) \leq \frac{1}{1-\epsilon} + m(0) (1-\epsilon)t.
\]
Since the function on the left side decreases to zero 
as ${\displaystyle t\to -\frac{1}{m(0)}\frac{1}{1+\epsilon}}$ and
since the function on the right side decreases to zero 
as ${\displaystyle t\to -\frac{1}{m(0)}\frac{1}{(1-\epsilon)^2}}$,
therefore, $q(t)\to 0$, and hence (see \eqref{def:m}) $m(t)\to-\infty$ as $t\to T-$, 
where $T$ satisfies \eqref{E:T}.
Note on the other hand that \eqref{claim:v0} dictates that $v_0(t;x)$ remains bounded 
for all $t\in [0,T']$, $T'<T$, for all $x\in \mathbb{R}$.
To summarize, $\inf_{x\in\mathbb{R}}\partial_xu(x,t)\to -\infty$ as $t\to T-$
but $u(x,t)$ is bounded for all $x\in\mathbb{R}$ for all $t\in [0,T)$, namely wave breaking.
This completes the proof.

\

\section{Proof of Theorem \ref{thm:1/3}}\label{sec:1/3}

Recall the notation of the previous section. We are done if 
\[
|K_1(t;x)|<\epsilon^2m^2(t) \qquad\text{for all $t\in[0,T)$ for all $x\in\mathbb{R}$}.
\]
Note from \eqref{A3:m1} and the Sobolev inequality that 
\[
|K_1(0;x)|=|\mathcal{H}\Lambda^\alpha\phi'(x)|\leq \|\phi\|_{H^{3/2+\alpha+}}<\epsilon^2m^2(0)
\qquad \text{for all $x\in\mathbb{R}$}.
\]
Suppose on the contrary that $|K_1(T_1;x)|=\epsilon^2m^2(T_1)$ 
for some $T_1\in(0,T)$ for some $x\in\mathbb{R}$. 
By continuity, without loss of generality, we may assume \eqref{I:K1}. 
We then rerun the argument in the previous section to arrive at that 
Lemma~\ref{lem:S}, Lemma~\ref{lem:q} and Lemma~\ref{lem:qs} hold. 

\

We claim that
\begin{align}
\|v_0(t)\|_{L^\infty}=&\|u(t)\|_{L^\infty}<C_0,\label{claim3:v0}\\
\|v_1(t)\|_{L^\infty}=&\|\partial_xu(t)\|_{L^\infty}<C_1q^{-1}(t)\label{claim3:v1}
\end{align}
for all $t\in[0,T_1]$, where $C_0$ and $C_1$ satisfy \eqref{I:C012}.
Note from \eqref{I:C012} and Lemma~\ref{lem:q} that 
\begin{align*}
\|v_0(0)\|_{L^\infty}=&\|\phi\|_{L^\infty}<C_0,\\
\|v_1(0)\|_{L^\infty}=&\|\phi'\|_{L^\infty}<C_1=C_1q^{-1}(0).
\end{align*}
Suppose on the contrary that either \eqref{claim3:v0} or \eqref{claim3:v1} fails 
at $T_2$ for some $T_2\in(0,T_1)$. 
By continuity, without loss of generality, we may assume that 
\begin{align}
\|v_0(t)\|_{L^\infty}\leq& C_0,\notag\\
\|v_1(t)\|_{L^\infty}\leq& C_1q^{-1}(t)\label{I3:v1}
\end{align}
for all $t\in[0,T_2]$.

\

For $n=0$, we rerun the argument in the previous section and we use \eqref{A3:m1} to deduce that 
\eqref{claim3:v0} holds throughout the interval $[0,T_2]$.

\

For $n=1$, we follow the proof of \eqref{eq:Kn} but use H\"older's inequality to show that
\begin{align}\label{eq:K1}
|K_1(t;x)|
=&\Big|\frac{1}{\alpha}\delta^{-\alpha}((\partial_xu)(X(t;x)-\delta,t)-(\partial_xu)(X(t;x)+\delta,t)) \notag \\
\qquad&+\frac{1}{\alpha}\int_{|y|<\delta}\frac{1}{|y|^\alpha}(\partial_x^2u)(X(t;x)-y,t)~dy \notag \\
\qquad&+\int_{|y|>\delta}\frac{\text{sgn}(y)}{|y|^{1+\alpha}}
((\partial_xu)(X(t;x),t)-(\partial_xu)(X(t;x)-y,t))~dy\Big|\notag \\
\leq &\frac{6}{\alpha}\delta^{-\alpha}\|v_1(t)\|_{L^\infty}
+\frac{1}{\alpha}\sqrt{\frac{2}{1-2\alpha}}\delta^{1/2-\alpha}\|v_2(t)\|_{L^2}
\end{align}
for all $t\in[0,T_2]$ for all $x\in\mathbb{R}$.

Differentiating \eqref{E:main} twice with respect to $x$ 
and integrating over $\mathbb{R}$ against $\partial_x^2u$, we promptly arrive at that 
\[
\frac12\frac{d}{dt}\int^\infty_{-\infty}(\partial_x^2u)^2~dx
+\int^\infty_{-\infty} \partial_x^2u\mathcal{H}\Lambda^\alpha \partial_x^2u~dx
+\int^\infty_{-\infty} (u(\partial_x^2u)(\partial_x^3u)+3(\partial_xu)(\partial_x^2u)^2)~dx=0.
\]
The second term on the left side vanishes by a symmetry argument
while an integration by parts leads to that
\[
\int^\infty_{-\infty} u(\partial_x^2u)(\partial_x^3u)~dx
=-\frac12\int^\infty_{-\infty}(\partial_xu)(\partial_x^2u)^2~dx.
\]
Consequently
\begin{align*}
\frac{d}{dt}\|\partial_x^2u(t)\|_{L^2}^2
=&-5\int^\infty_{-\infty}(\partial_xu)(\partial_x^2u)^2(x,t)~dx \\
\leq& -5m(t)\|\partial_x^2u(t)\|_{L^2}^2
=-5m(0)q^{-1}(t)\|\partial_x^2u(t)\|_{L^2}^2
\end{align*}
for all $t\in(0,T_2)$. The last equality uses \eqref{def:m}.
We then integrate it and use \eqref{I:s=1} to arrive at that
\begin{align*}
\log \|\partial_x^2u(t)\|_{L^2}^2-\log\|\phi''\|_{L^2}^2\leq& -5m(0)\int^t_0q^{-1}(\tau)~d\tau \\
\leq&\frac{5}{(1-\epsilon)^2}\Big(\log\frac{1}{1-\epsilon}-\log q(t)\Big)
\end{align*}
for all $t\in[0,T_2]$. Therefore
\begin{equation}\label{I:uxx}
\|\partial_x^2u(t)\|_{L^2}\leq \|\phi''\|_{L^2} (1-\epsilon)^{-5/2(1-\epsilon)^2} q^{-5/2(1-\epsilon)^2}(t)
\end{equation}
for all $t\in[0,T_2]$. 

To proceed, we use \eqref{eq:K1}, where $\delta(t)=q^{5/(1-\epsilon)^2-2}(t)=:q^\sigma(t)$,
and \eqref{I3:v1}, \eqref{I:uxx} to show that 
\begin{align*}
|K_1(t;x)|\leq& \frac{6}{\alpha}C_1\delta^{-\alpha}(t)q^{-1}(t)\\
&+\frac{1}{\alpha}\sqrt{\frac{2}{1-2\alpha}}\|\phi''\|_{L^2} (1-\epsilon)^{-5/2(1-\epsilon)^2}
\delta^{1/2-\alpha}(t)q^{-5/2(1-\epsilon)^2}(t) \\ 
\leq & \frac{1}{\alpha}(6C_1+C_2)q^{-1-\sigma\alpha}(t)
\leq  \frac{1}{\alpha}(6C_1+C_2)q^{-2}(t)
\end{align*}
for all $t\in [0,T_2]$ for all $x\in\mathbb{R}$. 
The second inequality uses \eqref{I:C012}, and 
the last inequality uses Lemma~\ref{lem:q} and that $\sigma\alpha<1$ (see Theorem~\ref{thm:1/3}). 
Note from \eqref{eqn:v1} that 
\[
\frac{dv_1}{dt}=-v_1^2-K_1(t;x)\leq |K_1(t;x)|.
\]
In case $K_1(t;x)\geq 0$, an integration then leads to that
\begin{align*}
v_1(t;x)\leq& \|\phi'\|_{L^\infty}+\frac{1}{\alpha}(6C_1+C_2)\int^t_0 q^{-2}(\tau)~d\tau \\
<&\frac12C_1q^{-1}(t)
-\frac{1}{\alpha}(6C_1+C_2)\frac{1}{(1-\epsilon)^3}\frac{1}{m(0)}(q^{-1}(t)-(1-\epsilon)) \\
<&\frac12C_1q^{-1}(t)-
\frac{1}{\alpha}(6C_1+C_2)\frac{1}{(1-\epsilon)^3}\frac{1}{m(0)}q^{-1}(t) \\
<&C_1q^{-1}(t).
\end{align*}
The second inequality uses \eqref{I:C012}, Lemma~\ref{lem:q} and \eqref{I:s>1}, 
and the last inequality uses \eqref{A3:m2}. 
In case $v_1(t;x)<0$, on the other hand, we may assume, without loss of generality, that 
$\|\phi'\|_{L^\infty}=-m(0)$ and \eqref{def:m} and \eqref{I:C012} imply that
\[
v_1(t;x)\geq m(t)=m(0)q^{-1}(t)>-\frac12C_1q^{-1}(t).
\]
Therefore \eqref{claim3:v1} holds throughout the interval $[0,T_2]$. 

\

The remainder of the proof is nearly identical to that in the previous section. Hence we omit the detail. 

\subsection*{Acknowledgment}
The authors thank the anonymous referees for their careful reading of the manuscript 
and for their many useful suggestions.
VMH is supported by the National Science Foundation grants DMS-1008885 and CAREER DMS-1352597, and by an Alfred P. Sloan Foundation fellowship. 

\bibliographystyle{amsalpha}
\bibliography{breakingBib}

\end{document}